\documentclass[11pt,a4paper]{article}

\usepackage{a4wide}
\usepackage{graphicx}
\usepackage{latexsym}
\usepackage{epsfig}
\usepackage{amssymb}
\usepackage{amstext}
\usepackage{amsgen}
\usepackage{amsxtra}
\usepackage{amsgen}
\usepackage{amsthm}
\usepackage{color}
\usepackage{pdfsync}

\usepackage{subcaption}

\def\eps{\varepsilon}

\usepackage{chngcntr}

\def\RR{\mathbb{R}}

\newtheorem{thm}{Theorem}[section]

\newtheorem{cor}[thm]{corollary}

\theoremstyle{definition}
\newtheorem{lemma}[thm]{Lemma}

\newtheorem{definition}[thm]{definition}

\numberwithin{equation}{section}

\usepackage[colorinlistoftodos]{todonotes}

\newcommand{\M}{{\cal M}}

\newcommand{\eq}{\eqref}

\newcommand{\beq}{\begin{equation}}
\newcommand{\eeq}{\end{equation}}

\title{On a Reaction--Cross--Diffusion System \\ Modelling the Growth of Glioblastoma}
\author{Martin Burger\thanks{Institut f\"ur Numerische und Angewandte Mathematik, Westf\"alische Wilhelms-Universit\"at (WWU) M\"unster. Einsteinstr. 62, D 48149 M\"unster, Germany. e-mail: martin.burger@wwu.de,\,patricia\_friele@gmx.de,\,jan.pietschmann@wwu.de} \and Patricia Friele$^*$ \and Jan-Frederik Pietschmann$^*$} 

\begin{document}

\maketitle
{\em Preprint: \today}
\begin{abstract}
\noindent We investigate a recently proposed cross-diffusion system modelling the growth of gliobastoma taking into account size exclusion both in the migration and proliferation process. In addition to degenerate nonlinear cross-diffusion the model includes reaction terms as in the Fisher-Kolmogorov equation and linear ones modelling transition between states of proliferation and migration.
We discuss the mathematical structure of the system and provide a complete existence analysis in spatial dimension one. The proof is based on exploiting partial entropy dissipation techniques and fully implicit time discretisations. In order to prove existence of the latter appropriate variational and fixed-point techniques are used, together with suitable a-priori estimates. Moreover, we review the existence of travelling waves and their relation to potential growth or decay of glioblastoma. Finally we provide extensive numerical studies in one and two spatial dimensions, including the effect of anisotropic diffusions as found in neural tissues. 
\end{abstract}

\section{Introduction}
The growth of tumours or more general the interaction of migration and proliferation, is a classical topic of mathematical modelling in cell biology. The standard route is taken via the Fisher-Kolmogorov or Fisher-KPP equation 
\begin{equation}\label{eq:FisherKPP}
\partial_t \rho = D \Delta \rho + \rho(1-\rho) 
\end{equation}
describing the evolution of the (relative) cell density $\rho=\rho(x,t)$. Here $\rho=1$ describes the maximal packing, at this point there is no further proliferation, resulting in the logistic reaction term on the right-hand side. There is a vast literature on the Fisher-Kolmogorov equation and its dynamical behaviour, in particular it is nowadays the standard example of a reaction-diffusion equation exhibiting travelling wave solutions. It can be viewed from stochastic or deterministic perspectives (cf. e.g. \cite{fisher1937wave,kolmogorov1937etude,bramson1983convergence,champneys1995algebra,freidlin1995wave,sherratt1998transition,kyprianou2004travelling}). 

While equation \eqref{eq:FisherKPP} is also used to model the growth behaviour of the brain tumours called glioblastoma, cf. \cite{belmonte2014effective,konukoglu2010extrapolating}, recently Gerlee et al. introduced a slightly refined model, cf. \cite{Gerlee2012,gerlee2016travelling}. Based on experimental evidence, they assume the existence of two distinct states that the cells in the tumour: Migration or proliferation. It is found that cells stop before they enter a cell division state, so there are two well separated states of cells. It is hence natural to derive a model for the individual densities of proliferating cells (denoted by $p$) and migrating cells (denoted by $m$) with a certain transition between these states. The starting point, cf. \cite{Gerlee2012}, is an individual-based model on an equally spaced lattice, where cells occupy one lattice site and can move to their neighbours if not occupied. Moreover, proliferation is carried out only if one of the neighbouring cells is unoccupied, which models the need for appropriate space for the two daughter cells. In addition, cells can randomly change their state between migration and proliferation with certain rates and also die at a certain rate. When carrying out a second order expansion in the lattice site this yields a reaction-diffusion system for $p$ and $m$ (with $\rho=p+m$ their sum), posed on a domain $\Omega \subset \RR^n$, $n=1,2,3$:
\begin{align}
\partial_t p&=D_{\alpha}(1-\rho)\Delta p + R_p(p,m),\label{eq:pnd}\\
\partial_t m &= D_{\nu}\left((1-\rho)\Delta m+m\Delta \rho\right) + R_m(p,m). \label{eq:mnd}
\end{align}
Here the reaction terms, which describe the transition between migration and proliferation, cell death, and birth due to proliferation are given by
\begin{align}\label{eq:reacp}
 R_p(p,m) &= -(\mu+q_m)p+q_pm + \alpha p(1-\rho),\\ \label{eq:reacm}
 R_m(p,m) &= -(\mu+q_p)m + q_mp,
\end{align}
where $q_p$ and $q_m$ denote the transition rate from either migrating to proliferating cells or vice versa. The constant $\mu > 0$ models the apoptosis (cell death) rate while $\alpha$ denotes the rate of cell division of species $p$. With $\rho$ being the density of all cells, the term $(1-\rho)$ reflect the fact that only a finite number of cells can occupy a given volume. We consider this system supplemented with no flux boundary conditions given by
\begin{align}\label{eq:bcs}
 \partial_n p = 0\text{ and } (1-p)\partial_n m + m \partial_n p = 0\text{ on } \partial\Omega.
\end{align}

Our aim in this paper is to provide a further mathematical understanding of the reaction-cross-diffusion system  \eqref{eq:pnd}, \eqref{eq:mnd} beyond the preliminary investigation in \cite{Gerlee2012}, based on calculating stationary solutions, linear stability, travelling waves and numerical solutions in spatial dimension one. We will give a rigorous existence proof for the system in the one-dimensional case, which needs a combination of entropy-type techniques for the nonlinear cross-diffusion parts with suitable techniques previously used for reaction-diffusion equations. We also revisit the computation of stationary solutions and their stability, since we find slightly different signs compared to the ones given in \cite{Gerlee2012}. In particular a necessary condition on the growth rate relative to the apoptosis- and the transition rates in order to obtain a non-trivial solution in addition to the zero density state is observed. This is of course related to a natural biological question of growth or dissolution of the tumour and becomes even more relevant when further treatment strategies are to be incorporated into the modelling. In addition, the travelling wave solutions give an idea of the type and speed of tumour expansion.

In order to obtain a more realistic scenario, we also provide a numerical investigation of the dynamics in two spatial dimensions. There we can also observe how a too small defect in a tumour (due to medical treatment) will be filled with cancerous tissue and study the effects of anisotropic spatially dependent diffusivity, which is naturally found in neural tissue (cf. \cite{Painter2013,engwer2015glioma}). 

The remainder of the paper is organised as follows: In section 2 we discuss some preliminary notations and definition and state the main theoretical result, whose proof is detailed in section 3. section 4 is devoted to a discussion of asymptotics with respect to time and parameters. section 5 provides a numerical study including multidimensional examples, finally an outlook to open problems for further research is given in section 6.

\section{Preliminaries and main results}
In the following we will set $\bar q_m = q_m + \mu$ and $\bar q_p = q_p + \mu$ to shorten the notation. Since our main results hold in one spatial dimension only, we explicitly state \eqref{eq:pnd}--\eqref{eq:mnd} for this case
\begin{align}
\partial_t p&=D_{\alpha}(1-\rho)\partial_{xx}p + R_p(p,m),\label{eq:p}\\
\partial_t m &= D_{\nu}\left((1-\rho)\partial_{xx}m + m\partial_{xx}\rho\right) + R_m(p,m), \label{eq:m}
\end{align}
and define the set
\begin{equation}\label{eq:M}
\mathcal{M}=\lbrace (p,m)\in L^2(\Omega)\times L^2(\Omega)\; |\; p,m\geq 0,\,p+m\leq 1\rbrace.
\end{equation}
Our definition of solutions is given as follows:
\begin{definition}[Strong solution]\label{def:weak}
Two functions $p,m:(0,T)\rightarrow \mathcal{M}\cap H^1(\Omega)$ are strong solutions to \eqref{eq:p}--\eqref{eq:m},  if
\begin{align*}
 \sqrt{(1-\rho )}\partial_{xx} p , \;\sqrt{(1-\rho )}\partial_{xx} m ,\;\sqrt{m }\partial_{xx} \rho, \partial_t p, \partial_t m \in L^\infty((0,T);L^2(\Omega)),
\end{align*}
and almost everywhere in $\Omega$ we have
\begin{align} 
(1-\rho) \partial_t p&=  D_\alpha  (1-\rho)^2 \partial_{xx}p 
+ (1-\rho)R_p(p,m) \label{eq:pweak}\\
 (1-\rho)\partial_t m &= D_\nu  (1-\rho)((1-\rho)\partial_{xx}m-m\partial_{xx}\rho)+ (1-\rho)R_m(p,m)   \label{eq:mweak}
\end{align}
\end{definition}
Note that compared to the usual strong form we multiplied the equations by $(1-\rho)$. This will be crucial for our existence results later on and is motivated by the fact that for $\rho=1$ the maximal packing limit is reached, hence the evolution is not well-defined. Similar reasoning has been carried out for weak solutions of cross-diffusion equations with size exclusion (cf. \cite{bruna2017cross}).

The main result of this paper is the following theorem
\begin{thm}[Existence of strong solutions]\label{thm:main}
For given initial data $(p_0,m_0) \in \mathcal{M}\cap H^1(\Omega)$ there exists a strong solution $(p,m)$ to \eqref{eq:p}--\eqref{eq:m}, in the sense of definition \ref{def:weak} with
\begin{align*}
p,\,m \in L^2(0,T;H^1(\Omega)), \text{ and }\; \partial_t p,\, \partial_t m \in L^2(0,T;L^2(\Omega)),
\end{align*}
Furthermore, the initial data is attained as follows
$$
\lim_{t\rightarrow0}\left\Vert p(\cdot,t)-p_0\right\Vert_{L^2(\Omega)}=\lim_{t\rightarrow0}\left\Vert m(\cdot,t)-m_0\right\Vert_{L^2(\Omega)}=0.
$$
\end{thm}
As usual the main difficulty in proving the theorem is the derivation of a priori estimates. Since we are dealing with a system, maximum principles are not available. Moreover, as the diffusion is degenerate, also energy estimates obtained by testing the equations with their respective solution do not work. In similar systems of cross-diffusion equations, this can be overcome by exploiting a formal gradient flow structure of the system. Unfortunately, in our case only the second equation \eqref{eq:m} can we written in such a way so that only a partial entropy structure is available. However, it will turn out that this asymmetric structure of our problem allows for a different approach, namely testing the equations with linear combinations of second derivatives of their solutions which eventually yields the needed estimates.
\subsection{Partial gradient flow structure}
As mentioned above, the system only features a partial entropy structure. Rewriting equation \eqref{eq:m} in divergence form, we obtain
\begin{align}\label{eq:ment}\notag
\partial_t m &= D_\nu((1-\rho)\partial_{xx}m+m\partial_{xx}\rho)+R_m(p,m)\\ 
&=D_\nu\partial_x((1-\rho)\partial_xm+m\partial_x \rho)+R_m(p,m)\\ \notag
&=D_\nu\partial_x \left( m(1-\rho)\partial_x(\partial_m\mathcal{E}_p(m,p))\right)+R_m(p,m).
\end{align}
Here, the term $m(1-\rho)$ is a nonlinear mobility and $\mathcal{E}_p$ denotes the entropy functional
\begin{align}
 \mathcal{E}_{p}(m) = \int_{\Omega} e_p(m)\;dx\;\;\text{ with }\;\;e_p(m) = m\log m + (1-\rho)\log (1-\rho),
\end{align}
where $e_p(m)$ is the corresponding entropy density. This structure also allows us to define an entropy (or dual) variable as 
\begin{align}\label{eq:defent}
 u := e_p'(m) = \partial_m\mathcal{E}_p(m) = \log\left(\frac{m}{1-\rho}\right).
\end{align}
Inverting this relation we obtain
\begin{align}\label{eq:entinv}
 m(u) = (e_p')^*(u) = \frac{e^u(1-p)}{1+e^u},
\end{align}
where $(e_p')^*$ denotes the convex conjugate. Note that as long as $0\le p \le 1$ this relation implies $0\le p + m(u) \le 1$, a fact which was first used in \cite{Burger2010} and is sometimes called \emph{boundedness by entropy principle}, cf. \cite{Juengel2015}. It will serve as a valuable tool to prove additional a priori bounds later on.

\section{Proof of Theorem \ref{thm:main}}
The proof of our main theorem consists of several main steps: First, we approximate system \eqref{eq:p}--\eqref{eq:m} by means of an implicit time discretisation, additional regularisation and the use of dual variables. The resulting system is still nonlinear and existence of iterates is shown using Schauder's fixed point theorem. Next, we derive a priori estimates which finally allow us to pass to the limit as the last step.
\subsection{Approximation}
Given $N\in \mathbb{N}$, we divide the interval $(0,T)$ into subintervals $(k\tau,(k+1)\tau]$, $k=0,\ldots,N-1$, where $\tau=\frac{T}{N}$ is the time step size. Then, for given functions $(p_k(x),m_k(x))\in \mathcal{M}$ we perform an implicit Euler discretisation of \eqref{eq:p} and \eqref{eq:ment} which yields
\begin{align}\label{eq:ptaupre}
\frac{p_{k+1}-p_k}{\tau} &= D_{\alpha}(1-p_{k+1}-m_{k+1})\partial_{xx}p_{k+1}+ R_p(p_{k+1},m_{k+1}),
\\ \label{eq:mtaupre}
\frac{m(u)-m_k}{\tau} &= D_\nu\partial_x(m_k(1-p_{k+1}-m_k)\partial_xu)+R_m(p_{k+1},m(u)),\\
m_{k+1} &= m(u).\label{eq:utaupre}
\end{align}
As a second step, we carry out and additional regularisation using the time step size $\tau$ as a parameter, see \cite{Lane2015,Berendsen2017} for a similar approach. We divide \eqref{eq:ptaupre} by the factor $1-p_{k+1}-m_{k+1}$ and introduce the additional regularisation in the denominator of the last term. Equation \eqref{eq:utaupre} is modified by adding the term $\tau\partial_{xx}u$. 
Inserting the definition of the reaction terms \eqref{eq:reacp}--\eqref{eq:reacm}, we arrive at
\begin{align} \label{eq:ptau}
0&=D_\alpha\partial_{xx}p_{k+1}+\alpha p_{k+1}+\bar q_m \\
&+\frac{\bar q_m(1-\tilde m_{k+1})+c_1m_{k+1}}{1-p_{k+1}-m_{k+1}}
+\frac{(q_p+c_1)m_{k+1} +1/\tau(p_k-p_{k+1})}{1-p_k-m_{k+1}+\tau}\label{eq:mtau}\\
0&= D_\nu\partial_x\left((\tilde m(1-\tilde p-\tilde m)+\tau)\partial_x u\right)-\bar q_pm(u)+q_mp_{k+1}\\ \notag
&-\frac{1}{\tau}(m(u)-m_k),\\
m_{k+1} &= m(u),\label{eq:utau}
\end{align}
where $c_1$ is a positive constant to be determined later. The resulting system \eqref{eq:ptau}--\eqref{eq:utau} is still nonlinear and the main result of this section is the following theorem:
\begin{thm}\label{thm:exiterates} 
For any $(p_k,m_k)\in\mathcal{M}$, there exists a tuple
\begin{align*}
(p_{k+1},m_{k+1})\in\mathcal{M} \cap H^1(\Omega)
\end{align*}
being a weak solution to \eqref{eq:ptau}--\eqref{eq:utau}.
\end{thm}
The proof is based on Schauder's fixed point theorem. We start by linearising equations \eqref{eq:ptau}--\eqref{eq:utau} as follows: Given $(\tilde p,\tilde m) \in \mathcal{M}$, we have
\begin{align} \label{eq:ptaulin}
0&=D_\alpha\partial_{xx}p_{k+1}+\alpha \tilde p+\bar q_m \\
&+\frac{\bar q_m(1-\tilde m)+c_1 \tilde m}{1-\tilde p-\tilde m}
+\frac{(q_p+c_1)\tilde m +1/\tau(p_k-p_{k+1})}{1-\tilde p- \tilde m+\tau}\notag\\ \label{eq:mtaulin}
0&= D_\nu\partial_x\left((m_k(1-p_{k+1}-m_k)+\tau)\partial_x u\right)-\bar q_pm(u)+q_mp_{k+1} -\frac{1}{\tau}(m(u)-m_k),\\
m_{k+1} &= m(u).\label{eq:utaulin}
\end{align}
We  define the fixed point operator $\mathcal{F}: \mathcal{M} \to \mathcal{M}$ as the operator which maps $(\tilde p, \tilde m)$ to $(p_{k+1},m_{k+1})$ being solutions of \eqref{eq:ptaulin}--\eqref{eq:utaulin}. Since the linearised system is decoupled, this operator can be decomposed as $\mathcal{F} = \mathcal{J} \circ \mathcal{H} \circ \mathcal{G}$ with 
\begin{align}
 (\tilde p,\tilde m) \stackrel{\mathcal{G}}{\longrightarrow} (p_{k+1},\tilde m) \stackrel{\mathcal{H}}{\longrightarrow} (p_{k+1},u)\stackrel{\mathcal{J}}{\longrightarrow} (p_{k+1},m_{k+1}),
\end{align}
where $\mathcal{G}:\M \to \M$, $\mathcal{H} : \M \to \overline{\mathcal{M}}$, $\mathcal{J} : \overline{\mathcal{M}} \to \M$ and with
\begin{equation}\label{eq:Mbar}
\overline{\mathcal{M}}=\lbrace (p,u)\in L^2(\Omega)\times L^2(\Omega) \;|\; 0<p<1 \rbrace.
\end{equation}
Next we will show that each of these operators is well defined and continuous and that $\mathcal{F}$ is compact. For brevity we will neglect the indices $k+1$ in the following and simply denote the new iterates by $p$ and $m$.
\begin{lemma}\label{lem:G} The operator $\mathcal{G}:\M \to \M$ which maps a tuple $(\tilde p, \tilde m)$ to $(p, \tilde m) \in \M$, where $p$ is the unique solution of \eqref{eq:ptaulin} is well-defined and continuous with respect to the strong topology in $L^2(\Omega)$. Furthermore, it is compact in its first component and continuous in its second.
\end{lemma}
\begin{proof}
The proof is based on a variational argument. In fact, \eqref{eq:ptaulin} is the Euler--Lagrange equation of the functional 
\begin{subequations}
\begin{align}\label{eq:G1}
G_{(\tilde{p},\tilde{m})}(p):&=\frac{D_\alpha}{2}\int_\Omega\left(\partial_x p\right)^2dx +\frac{1}{2\tau}\int_\Omega\frac{(p-p_k)^2}{1-\tilde{p}-\tilde{m}+\tau}dx \\ \label{eq:G2}
&-\alpha\int_\Omega\tilde{p}pdx -\bar q_m\int_\Omega pdx -\int_\Omega\frac{(q_p+c_1)\tilde{m}}{1-\tilde{p}-\tilde{m}+\tau}pdx \\ \label{eq:G3}
& -\int_\Omega\left(\tilde{m}(c_1-\bar q_m)+\bar q_m\right)\log(1-\tilde{m}-p)dx.
\end{align}
\end{subequations}
We employ the direct method of calculus of variations to show the existence of a unique minimiser $p \in \M \cap H^1(\Omega)$ of $G$. To this end, we have to show coercivity, lower semincontinuity and strict convexity of $G$, cf. \cite{braides2002gamma}. Since both the second term in \eqref{eq:G1} and \eqref{eq:G3} are non-negative, they can be neglected when estimating the functional from below. Terms in \eqref{eq:G2}, being linear in $p$, can be estimated using Cauchy-Schwarz and the weighted Young inequality to obtain
\begin{align*}
G_{(\tilde{p},\tilde{m})}(p)&\geq \frac{D_\alpha}{2}|p|_{H^1(\Omega)}^2-\left(C_1\epsilon\|p\|_{L^2(\Omega)}^2+\frac{1}{\epsilon}C_2\right)\\
&\geq \left(\frac{D_\alpha}{2}-\epsilon C_1\right)\|p\|_{H^1(\Omega)}^2- \frac{1}{\eps}C_2,
\end{align*}
where we used Friedrichs inequality in the last step, $|\cdot|_{H^1(\Omega)}$ denotes the $H^1$-seminorm and where the constants $C_1$ and $C_2$ depend only on $q_m,\,q_m\,\mu,\,\alpha,\,\tau$ and $D_\alpha$. Thus choosing $\eps$ small enough such that $\frac{D_\alpha}{2}-\epsilon C_1 > 0$, we also conclude the coercivity of $G$ with respect to $H^1(\Omega)$. Lower semincontinuity of the quadratic and linear terms in \eqref{eq:G1}--\eqref{eq:G2} follows from their (strict) convexity \cite{braides2002gamma}. The same holds true for the logarithmic term which has a positive second derivative w.r.t $p$ for $c_1 \ge \bar q_m$ and is thus also convex. Therefore, there exists a minimiser which is unique due to the strict convexity of the quadratic terms which implies convexity of the whole functional. The logarithmic term ensures that $(p,\tilde m) \in \M$ still holds since otherwise, the functional would be infinite.\\
To show the continuity of the operator $\mathcal{G}$, we take arbitrary sequences $\tilde{p}_n\rightarrow \tilde{p}$ and $\tilde{m}_n\rightarrow\tilde{m}$ in $L^2(\Omega)$ and obtain $p_n$, the corresponding sequence of minimisers of $G_{(\tilde{p}_n,\tilde{m}_n)}(p_n)$. Due to the coercivity, $p_n$ is uniformly bounded in $H^1(\Omega)$ and thus admits a weakly converging subsequence $p_{n_k}\rightharpoonup p$ in $H^1(\Omega)$. To pursue, we note that $G_{(\tilde{p}_n,\tilde{m}_n)}$ is also lower semincontinuous w.r.t to $(\tilde m, \tilde p)$, which is again a consequence of convexity for \eqref{eq:G3} and the definition of $\M$ for the second term in \eqref{eq:G1}. The last term in \eqref{eq:G2} is even continuous, due to the uniform boundedness of the denominator. Thus,
\begin{equation}
G_{(\tilde{p},\tilde{m})}(p) \leq \liminf_{n\rightarrow \infty}G_{(\tilde{p}_n,\tilde{m}_n)}(p_{n_k})\leq \liminf_{n\rightarrow\infty}G_{(\tilde{p}_n,\tilde{m}_n)}(\overline{p})= G_{(\tilde{p},\tilde{m})}(\overline{p}) \quad \forall\, \overline{p}, \notag
\end{equation}
which ensures that the limit is again a minimiser. Standard Sobolev embedding and the uniqueness of minimisers then yields the convergence of the whole sequence $p_n$ strongly in $L^2$. Finally, compactness follows from the compact embedding $H^1(\Omega) \hookrightarrow L^2(\Omega)$.
\end{proof}
\begin{lemma}\label{lem:H} The operator $\mathcal{H} : \M \to \overline{\mathcal{M}}$ which maps a tuple $(p, \tilde m)$ to $(p, u)$, where $u$ is the unique solution of \eqref{eq:mtaulin} is well-defined and continuous with respect to the strong topology in $L^2(\Omega)$ and compact in its second component.
\end{lemma}
\begin{proof}
Again, we have a variational principle and the proof basically follows the one of the previous lemma. Equation \eqref{eq:mtaulin} is the Euler--Lagrange equation to 
\begin{subequations}
\begin{align}\label{eq:H1}
H_{(p,\tilde{m})}(u):=&\frac{D_\nu}{2}\int_\Omega\left(\tilde{m}(1-p-\tilde{m})+\tau\right)(\partial_x u)^2dx +\bar q_p\int_\Omega e_p^*(u)dx \\ \label{eq:H2}
&-q_m\int_\Omega pudx -\frac{1}{\tau}\int_\Omega\tilde{m}udx+\frac{1}{\tau}\int_\Omega e_p^*(u)dx.
\end{align}
\end{subequations}
First we note that $e_p^*(u)$ is non-negative being the integral of a non-negative quantity. Thus, using the definition of $\M$, coercivity follows as above. Furthermore, the first term in \eqref{eq:H1} is strictly convex while the terms involving $e_p^*$ are convex as well. The remaining terms are linear which yields the lower semincontinuity of the functional and hence the existence of a unique minimiser. To prove continuity of $\mathcal{H}$, first note that the function $e_p^*$ is in fact continuous both with respect to $p$ and $u$ strongly in $L^2(\Omega)$ which can be shown easily using the boundedness and continuity of \eqref{eq:entinv}. To show lower semincontinuity of the first term in \eqref{eq:H1}, we consider the sequence 
\begin{align*}
r_n := \sqrt{\tilde{m}_n(1-p_n-\tilde{m}_n) + \tau}\nabla u_n.
\end{align*}
with $\tilde{m}_n\rightarrow\tilde{m}$ and $p_n\rightarrow p$. Then
\begin{align*}
&\left\|\sqrt{\tilde{m}(1-p-\tilde{m})}-\sqrt{\tilde{m}_n(1-p_n-\tilde{m}_n)}\right\|_{L^2(\Omega)}^2 \leq \left\|\sqrt{| \tilde{m}-\tilde{m}_n-\tilde{m}p+\tilde{m}_np_n-\tilde{m}^2+\tilde{m}^2_n|}\right\|_{L^2(\Omega)}^2 \\
&\qquad\qquad \leq \|\tilde{m}-\tilde{m}_n\|_{L^1(\Omega)} +\|\tilde{m}_np_n-\tilde{m}p\|_{L^1(\Omega)} + \|\tilde{m}_n^2-\tilde{m}^2\|_{L^1(\Omega)}\\
&\qquad\qquad \leq \|\tilde{m}-\tilde{m}_n\|_{L^1(\Omega)} +\underbrace{\|\tilde{m}_n\|_{L^1(\Omega)}}_{<\infty}\| p_n-p\|_{L^1(\Omega)} \\
&\qquad\qquad +\underbrace{\| p\|_{L^1(\Omega)}}_{<\infty}\|\tilde{m}_n-\tilde{m}\|_{L^1(\Omega)} + \|\underbrace{\tilde{m}+\tilde{m}_n}_{\leq 2}\|_{L^1(\Omega)}\|\tilde{m}-\tilde{m}_n\|_{L^1(\Omega)}\\
&\qquad\qquad \leq c\|\tilde{m}-\tilde{m}_n\|_{L^1(\Omega)} + \bar q_m\| p_n-p\|_{L^1(\Omega)},
\end{align*} 
where we used the reverse triangle inequality applied to the square root function and the definition of $\M$. Thus, the first part of $r_n$ converges strongly in $L^2(\Omega)$ while the second part converges weakly in $L^2$. Therefore,
\begin{align*}
r_n \rightharpoonup r \text{ in }L^1(\Omega).
\end{align*}
On the other hand, since $\sqrt{\tilde{m}_n(1-p_n-\tilde{m}_n) + \tau} \in L^\infty(\Omega)$, $r_n$ is bounded in $L^2(\Omega)$ and there exists a weakly converging sub-sequence. This implies 
\begin{align*}
r_n \rightharpoonup r \text{ in }L^2(\Omega),
\end{align*}
which, together with the weak lower semincontinuity of the norm completes the argument. Compactness again follows from the embedding $H^1(\Omega) \hookrightarrow L^2(\Omega)$.
\end{proof}
We are now in a position to prove that the operator $\mathcal{F}$ indeed admits a fixed--point.
\begin{proof}[Proof of Theorem \ref{thm:exiterates}]
Since
\begin{align*}
 0 \le p + m(u) = \frac{e^u(1-p) + p(1+e^u)}{(1+e^u)} = \frac{p + e^u}{1+e^u}\text{ for all } p \in \overline{\mathcal{M}},
\end{align*}
the operator $\mathcal{J}$ as defined in \eqref{eq:entinv} really maps from $\overline{\mathcal{M}}$ into $\M$. Since it is also continous, lemmata \ref{lem:G} and \ref{lem:H} imply that the operator $\mathcal{F}:\mathcal{M}\rightarrow\mathcal{M}$ is in fact compact, continous and self-mapping. Thus, Schauder's fixed point theorem yields the existence of a unique fixed point $(p_{k+1},m_{k+1})\in\mathcal{M}$ being the solution to \eqref{eq:ptau}--\eqref{eq:utau}.
\end{proof}
\subsection{A priori estimates}
To pass to the limit $\tau \to 0$ in \eqref{eq:ptau}--\eqref{eq:utau} and to obtain a weak solution in the sense of definition \ref{def:weak}, additional a priori estimates independent of $\tau$ are needed. Only part of the system has gradient flow structure, yet even for systems with size exclusion possessing this property, the dissipation of the entropy will usually not yield $H^1$ bounds on each density but only weaker estimates. Here, however, the asymmetric structure of the system will allow us to use a different approach, namely to test the discretised equations with linear combinations of the second derivatives of $p$ and $m$. To this end, on 
$$[0,\,T]=\bigcup_{k=0}^{N-1}\left(k\tau,\,(k+1)\tau\right),\quad N = \frac{T}{\tau},$$
we define the piecewise constant (in time) functions $p_\tau$ and $m_\tau$ as
\begin{align*}
p_\tau(x,t)=p_{k+1}(x),\,m_\tau(x,t)=m_{k+1}(x) \quad \forall \,t\in \left(k\tau,\,(k+1)\tau\right). 
\end{align*}
Then we have:
\begin{lemma}[A priori estimates]\label{lem:apriori}
For arbitrary $(p_0,m_0)\in \M \cap H^1(\Omega)$ and $\tau \in (0,T)$ 
we denote by $(p_\tau,m_\tau)\in\mathcal{M}$ the functions obtained by interpolating the solution of system \eqref{eq:ptau}--\eqref{eq:utau}. Then, there exists a constant $C>0$, independent of $\tau$, such that the estimates
\begin{align}
\|\partial_x p_\tau\|_{L^\infty((0,T);L^2(\Omega))} + \|\partial_x m_\tau\|_{L^\infty((0,T)L^2(\Omega))} &\leq C(p_0,m_0) \label{4.15}\\
\| p_\tau\|_{L^\infty((0,T);H^1(\Omega))} + \| m_\tau\|_{L^\infty((0,T);H^1(\Omega))} &\leq C(p_0,m_0), \label{4.16}\\
\|\sqrt{(1-\rho_\tau)}\partial_{xx} p_\tau\|_{L^\infty((0,T);L^2(\Omega))} &\le C(p_0,m_0),\\
\|\sqrt{(1-\rho_\tau)}\partial_{xx} m_\tau\|_{L^\infty((0,T);L^2(\Omega))} + \|\sqrt{m_\tau}\partial_{xx} \rho_\tau\|_{L^\infty((0,T);L^2(\Omega))} &\le C(p_0,m_0),
\end{align}
hold.
\end{lemma}
\begin{proof}
For simplicity, we will work with the discrete iterates and neglect the additional regularisation introduced in \eqref{eq:ptau}--\eqref{eq:mtau} and consider both equations in primal variables, i.e.
\begin{align}\label{eq:papriori}
\frac{p_{k+1}-p_k}{\tau}=& D_\alpha(1-\rho_{k+1})\partial_{xx}p_{k+1} + \alpha p_{k+1}(1-\rho_{k+1})
-\bar q_m p_{k+1} +q_p m_{k+1} \\ \label{eq:mapriori}
 \frac{m_{k+1}-m_k}{\tau} =& D_\nu \left((1-\rho_{k+1})\partial_{xx}m_{k+1}+m_{k+1}\partial_{xx}\rho_{k+1}\right) -\bar q_p m_{k+1} +q_m p_{k+1}
\end{align}
To obtain the desired estimates, we multiply equation \eqref{eq:papriori} first by $-2\frac{D_\nu}{D_\alpha}\partial_{xx}p_{k+1}$ and then by $-\frac{D_\nu}{D_\alpha}\partial_{xx}m_{k+1}$ and integrate over $\Omega$. To increase readability, we first neglect the reaction terms and consider them separately later on. Integrating by parts on the left hand side we obtain 
\begin{subequations}
\begin{align}\label{eq:aprioripp}
\frac{2D_\nu}{D_\alpha}\int_\Omega\!\! \frac{\partial_x(p_{k+1}-p_k)}{\tau}\partial_x p_{k+1}dx 
&= -2D_\nu\!\int_\Omega\!\! (1-\rho_{k+1})(\partial_{xx}p_{k+1})^2dx, \\
\frac{D_\nu}{D_\alpha}\int_\Omega \frac{\partial_x (p_{k+1}-p_k)}{\tau}\partial_x m_{k+1}dx 
&= -D_\nu\!\int_\Omega\!\! (1-\rho_{k+1})\partial_{xx}p_{k+1}\partial_{xx}m_{k+1}dx,  
\end{align}
\end{subequations}
Next we multiply \eqref{eq:mapriori} by $-\partial_{xx}\rho_{k+1}$ which gives
\begin{align}\label{eq:apriorimp}
\int_\Omega \frac{\partial_x (m_{k+1}-m_k)}{\tau}\partial_x \rho_{k+1}dx 
=& -D_\nu\!\int_\Omega\!\! (1-\rho_{k+1})\left((\partial_{xx}m_{k+1})^2 + \partial_{xx}m_{k+1}\partial_{xx}p_{k+1}\right)dx \notag\\
&- D_\nu\!\int_\Omega\!\! m_{k+1}(\partial_{xx}\rho_{k+1})^2dx 
\end{align}
Adding \eqref{eq:aprioripp}--\eqref{eq:apriorimp} results in 
\begin{align}
&\frac{1}{\tau}\int_\Omega \left[\partial_x(p_{k+1}-p_k)\frac{D_\nu}{D_\alpha}\left(2\partial_x p_{k+1} + \partial_x m_{k+1}\right) + \partial_x(m_{k+1}-m_k)\left(\partial_x p_{k+1} + \partial_x m_{k+1}\right)\right]dx \notag\\\label{eq:sumapriori}
=& -D_\nu\int_\Omega (1-\rho_{k+1})\left( (\partial_{xx}p_{k+1})^2 + (\partial_{xx}p_{k+1}+\partial_{xx}m_{k+1})^2\right) dx\\
&-D_\nu\int_\Omega m_{k+1}(\partial_{xx} \rho_{k+1})^2dx \le 0.
\end{align}
To estimate the left hand side of \eqref{eq:sumapriori}, we rewrite the expression under the integral as
\begin{equation}
\left(\begin{array}{c}\partial_x(p_{k+1}-p_k) \\ \partial_x (m_{k+1}-m_k)\end{array}\right)^T M \left(\begin{array}{c}\partial_x p_{k+1} \\ \partial_x m_{k+1}\end{array} \right) \label{4.24}
\end{equation}
with $M=\left(\begin{array}{cc} 2\frac{D_\nu}{D_\alpha} & \frac{D_\nu}{D_\alpha} \\ 1 & 1\end{array}\right) $. Since $M$ is a positive definite matrix, we can define, for $v\in \mathbb{R}^2$, the bilinear form $F(v)=\frac{1}{2}v^TMv > 0$ with derivative $F'(v)=Mv$. Applying the tangent inequality $F(x)-F(y)\leq (x-y)^T\nabla F(x)$, with $x=\left(\begin{array}{c}\partial_x p_{k+1} \\ \partial_x m_{k+1}\end{array}\right)$ and $y=\left(\begin{array}{c}\partial_x p_k \\ \partial_x m_k\end{array}\right)$ we obtain
\begin{align}
&\frac{1}{\tau}\int_\Omega \left[\partial_x(p_{k+1}-p_k)\frac{D_\nu}{D_\alpha}\left(2\partial_x p_{k+1} + \partial_x m_{k+1}\right) + \partial_x(m_{k+1}-m_k)\left(\partial_x p_{k+1} + \partial_x m_{k+1}\right)\right]dx \notag\\
\geq& F(x)-F(y) = \frac{1}{2}x^TMx -\frac{1}{2}y^TMy \notag\\
=& \frac{1}{2}\left((\partial_x p_{k+1})^2+\frac{1}{2}\partial_x m_{k+1}\partial_x p_{k+1} +(\partial_x m_{k+1})^2\right) - \frac{1}{2}\left((\partial_x p_k)^2 +\frac{1}{2}\partial_x p_k\partial_x m_k +(\partial_x m_k)^2\right)\notag\\
\stackrel{Young}{\geq}& \frac{(\partial_x p_{k+1})^2}{2}-\frac{(\partial_x p_{k+1})^2}{8} +\frac{(\partial_x m_{k+1})^2}{2}-\frac{(\partial_x m_{k+1})^2}{8} \notag\\
&- \frac{(\partial_x p_k)^2}{2}-\frac{(\partial_x p_k)^2}{8}-\frac{(\partial_x m_k)^2}{2}-\frac{(\partial_x m_k)^2}{8}\notag\\
=& \frac{3(\partial_x p_{k+1})^2}{8} - \frac{5(\partial_x p_k)^2}{8} + \frac{3(\partial_x m_{k+1})^2}{8}-\frac{5(\partial_x m_k)^2}{8}. \label{eq:lowerleft}
\end{align}
Thus combining \eqref{eq:sumapriori} and \eqref{eq:lowerleft}, we have
\begin{align*}
\frac{1}{\tau}\int_\Omega \frac{3(\partial_x p_{k+1})^2}{8}  + \frac{3(\partial_x m_{k+1})^2}{8} dx \le \frac{1}{\tau}\int_\Omega \frac{5(\partial_x p_k)^2}{8} + \frac{5(\partial_x m_k)^2}{8}dx. 
\end{align*}
To complete the proof, we need to address the reaction terms $R_p$ and $R_m$. Testing as above gives 
\begin{align*} 
 -&\frac{D_\nu}{D_\alpha}\int_\Omega R_p(p,m)(2\partial_{xx}p_{k+1}+ \partial_{xx}m_{k+1}) dx \\
 &=\frac{D_\nu\alpha}{D_\alpha}\!\int_\Omega\! (1-\rho_{k+1})\left( 2(\partial_x p_{k+1})^2+\partial_xp_{k+1}\partial_x m_{k+1}\right) - p_{k+1}\partial_x\rho_{k+1}(2\partial_x p_{k+1} + \partial_x m_{k+1})dx \notag\\
&-\frac{D_\nu\bar q_m}{D_\alpha}\!\int_\Omega\!2\left( \partial_x p_{k+1}\right)^2 + \partial_x p_{k+1}\partial_x m_{k+1}\mathrm{d}x +q_p\!\int_\Omega\! 2\partial_x m_{k+1}\partial_x p_{k+1} + (\partial_x m_{k+1})^2dx \notag\\
&\le \frac{D_\nu\alpha}{D_\alpha}\!\int_\Omega\! \frac{5}{2}(\partial_x p_{k+1})^2 + \frac{1}{2}(\partial_x m_{k+1})^2 - p_{k+1}\partial_x p_{k+1}\partial_x m_{k+1}dx\notag\\
&+ \frac{D_\nu\bar q_m}{D_\alpha}\!\int_\Omega\! \frac{5}{2}(\partial_x p_{k+1})^2 + \frac{1}{2}(\partial_x m_{k+1})^2 dx\notag\\
&+ \frac{D_\nu q_p}{D_\alpha}\!\int_\Omega\! (\partial_x p_{k+1})^2 + 2(\partial_x m_{k+1})^2 dx\notag\\
&\le 3\frac{D_\nu}{D_\alpha}(1+\alpha+\bar q_m + q_p)\int_\Omega\! (\partial_x p_{k+1})^2 + (\partial_x m_{k+1})^2 dx\notag
\end{align*}
and
\begin{align*} 
 -\int_\Omega R_m(p,m)\partial_{xx}\rho_{k+1}dx &=-\bar q_p\int_\Omega\! \partial_x m_{k+1}\partial_x \rho_{k+1} \;dx + q_m\int_\Omega \partial_x p_{k+1}\partial_x \rho_{k+1}dx.\\\nonumber
 &\le \frac{\bar q_p + q_m}{2}\int_\Omega\! (\partial_x p_{k+1})^2 + (\partial_x m_{k+1})^2 dx
\end{align*}
Thus, combining these estimates with those above, we finally arrive at
\begin{align*}
&\frac{1}{\tau}\int_\Omega (\partial_x p_{k+1})^2  + (\partial_x m_{k+1})^2 dx \le \frac{1}{\tau}\int_\Omega \frac{5(\partial_x p_k)^2}{3} + \frac{5(\partial_x m_k)^2}{3}dx\\
&+ \left(8\frac{D_\nu}{D_\alpha}(1+\alpha+\bar q_m + q_p) + \frac{4(\bar q_p + q_m)}{3}\right)\int_\Omega\! (\partial_x p_{k+1})^2 + (\partial_x m_{k+1})^2 dx.
\end{align*}
Introducing 
$\lambda = \left(8\frac{D_\nu}{D_\alpha}(1+\alpha+\bar q_m + q_p) + \frac{4(\bar q_p + q_m)}{3}\right)$, 
and choosing $\tau$ small enough so that $1- \lambda\tau > 0$, then employ a discrete version of Gronwall's Lemma, \cite[Proposition 3.1]{Emmrich99}, which yields the boundedness of the left hand side in terms of the initial data. Since the interpolation in time does not change the regularity in space, we arrive at the assertion. 
\end{proof}
As a corollary, we have the following estimates on the discrete time derivatives
%
\begin{cor}[]\label{cor:apriorit} For every $p_\tau,\,m_\tau$ satisfying \eqref{eq:ptau}--\eqref{eq:utau}, we have 
 \begin{equation} \label{4.36}
\frac{1}{\tau}\|p_\tau-\sigma_\tau p_\tau\|_{L^2(0,T;L^2(\Omega))} + \frac{1}{\tau}\|m_\tau-\sigma_\tau m_\tau\|_{L^2(0,T;L^2(\Omega))} \leq C,
\end{equation}
with a constant $C$ independent of $\tau$. 
\end{cor}
\begin{proof} This follows directly from the strong form of the equations in primal variables as the a priori estimates above yield the boundedness of the right hand side of the equations in the desired space.
\end{proof}
\subsection{Limit $\tau \to 0$}
The functions $p_\tau$, $m_\tau$ as defined in the previous section satisfy
\begin{align}\notag
\frac{1}{\tau}\int_0^T \int_\Omega (1-\rho_\tau)(p_\tau-\sigma_\tau p_\tau)\varphi_1\;dxdt =&  D_\alpha \int_0^T\int_\Omega (1-\rho_\tau)^2\partial_{xx} p_\tau\varphi_1 \,\mathrm{d}x\,dt\\\label{eq:pweaktau}
&+\int_0^T\int_\Omega (1-\rho_\tau)R_p(p_\tau,m_\tau) \varphi_1dxdt,\\ \notag
\frac{1}{\tau}\int_0^T\int_\Omega (1-\rho_\tau)(m_\tau-\sigma_\tau m_\tau)\varphi_2\;dxdt =&  D_\nu \int_0^T\int_\Omega \left((1-\rho_\tau)^2\partial_{xx} m_\tau+(1-\rho_\tau)m_\tau\partial_{xx}\rho_\tau\right)\varphi_2 \,\mathrm{d}x\,dt\\
&+ \int_0^T\int_\Omega (1-\rho_\tau)R_m(p_\tau,m_\tau)\varphi_2dxdt, \label{eq:mweaktau}
\end{align}
for every $\varphi_1,\,\varphi_2 \in L^2((0,T),H^1(\Omega))$. Note that due to the a priori estimates of Lemma \ref{lem:apriori}, all integrals above are well-defined.\\
In the remaining part of this section, we will show that as $\tau \to 0$, we recover a weak solution in the sense of definition \ref{def:weak}. Using the bounds of Lemma \ref{lem:apriori} and corollary \ref{cor:apriorit}, the Aubin-Lions-Simon Lemma, \cite{Lions1969}, yields
\begin{equation} \label{4.37}
p_\tau\rightarrow p,\quad m_\tau\rightarrow m,\quad \rho_\tau \rightarrow \rho\quad \text{in } \, L^2(0,\,T;\,L^2(\Omega)),
\end{equation}
for subsequences which we again label $p_\tau$ and $m_\tau$. This also implies the convergence of $1-\rho_\tau\rightarrow 1-\rho$ in $L^2(0,\,T;\,L^2(\Omega))$ and $\sqrt{1-\rho_\tau}\rightarrow \sqrt{1-\rho}$ in $L^4(0,\,T;\,L^4)$ and, using the continuous embedding, in $L^2(0,\,T;\,L^2(\Omega))$. We now treat each term in \eqref{eq:pweaktau}--\eqref{eq:mweaktau} separately, starting with the first term on the right hand side of \eqref{eq:pweaktau}. Integrating by parts, we have
\begin{align*}
 D_\alpha \int_0^T\int_\Omega (1-\rho_\tau)^2\partial_{xx} p_\tau \varphi_1 \,\mathrm{d}x\,dt &= - D_\alpha \int_0^T\int_\Omega (1-\rho_\tau)^2\partial_x p_\tau\partial_x\varphi_1 \,\mathrm{d}x\,dt\\
 &+ 2D_\alpha \int_0^T\int_\Omega(1-\rho_\tau)\partial_x\rho_\tau\partial_x p_\tau\varphi_1 \,\mathrm{d}x\,dt.
\end{align*}
For the first term we note that $(1-\rho_\tau)^2\partial_x\varphi_1$ converges, due to $0 \le \rho \le 1$ and Lebesgues dominated convergence theorem, strongly in $L^2(\Omega)$. For the second term, since $\partial_x \rho_\tau$ converges weakly in $L^2$, we will show the strong convergence of $(1-\rho_\tau)\partial_x p_\tau$ in $L^2$. Differentiating this term w.r.t. $x$ and $t$ (as we are dealing with piecewise constant functions in time, differentiating with respect to $t$ actually means taking shiftes differences, yet for the sake of presentation we make an abuse of notation and use $\partial_t$ instead) gives 
\begin{align}
\partial_x\big((1-\rho_\tau)\partial_x p_\tau\big) =& (1-\rho_\tau)\partial_{xx}p_\tau - \partial_x\rho_\tau\partial_x p_\tau, \label{eq:aubin1}\\
\partial_t\big((1-\rho_\tau)\partial_x p_\tau\big) =& -\partial_t\rho_\tau\partial_x p_\tau + (1-\rho_\tau)\partial_x\partial_t p_\tau \notag\notag\\
=& -\partial_t\rho_\tau\partial_x p_\tau + \partial_x\rho_\tau\partial_t p_\tau + \partial_x\big((1-\rho_\tau)\partial_t p_\tau\big) . \label{eq:aubin2}
\end{align}
Due to Lemma \ref{lem:apriori}, we see that the right hand side of \eqref{eq:aubin1} and thus $\partial_x\big((1-\rho_\tau)\partial_x p_\tau\big)$ is in fact bounded in $L^1(\Omega)$. As for \eqref{eq:aubin2}, corollary \ref{cor:apriorit} ensures that the first two terms on the right hand side are both $L^2((0,T);L^1(\Omega))$. By the same argument we see that the remaining term in bounded in $H^{-1}(\Omega)$. This yields the estimate
\begin{equation}
\|\sigma_\tau\left((1-\rho_\tau)\partial_x p_\tau\right)-(1-\rho_\tau)\partial_x p_\tau\|_{L^1(0,T;H^{-1}(\Omega))} + \|(1-\rho_\tau)\partial_x p_\tau\|_{L^2(0,T;W^{1,1}(\Omega))} \leq C \notag
\end{equation} 
We can now apply a variant of the Aubin-Lions-Simon Lemma for piecewise constant functions, \cite[Theorem 3]{Chen2013}, to obtain the strong convergence (of a subsequence) of $(1-\rho_\tau)\partial_x p_\tau$ in $L^2((0,T);L^2(\Omega))$ and thus the convergence of the first term on the right hand side of \eqref{eq:pweaktau}. In the same way, we obtain
\begin{align*}
(1-\rho_\tau)\partial_x m_\tau \partial_x\rho_\tau &\overset{\tau\rightarrow 0}{\rightarrow} (1-\rho)\partial_x m \partial_x\rho \quad \text{in} \; L^2(0,\,T;\,L^2(\Omega)), \\
(1-\rho_\tau)\partial_x p_\tau \partial_x m_\tau &\overset{\tau\rightarrow 0}{\rightarrow} (1-\rho)\partial_x p \partial_x m \quad \text{in} \; L^2(0,\,T;\,L^2(\Omega)), \\
(1-\rho_\tau)\partial_x m_\tau \partial_x m_\tau &\overset{\tau\rightarrow 0}{\rightarrow} (1-\rho)\partial_x m \partial_x m \quad \text{in} \; L^2(0,\,T;\,L^2(\Omega)), \\
m_\tau\partial_x \rho_\tau \partial_x p_\tau &\overset{\tau\rightarrow 0}{\rightarrow} m\partial_x \rho \partial_x p \quad \text{in} \; L^2(0,\,T;\,L^2(\Omega)), \\
m_\tau\partial_x \rho_\tau \partial_x m_\tau &\overset{\tau\rightarrow 0}{\rightarrow} m\partial_x \rho \partial_x m \quad \text{in} \; L^2(0,\,T;\,L^2(\Omega)), 
\end{align*}
which is enough to pass to the limit in the first term of \eqref{eq:mweaktau}. The strong convergence of $p_\tau, m_\tau$ and thus $\rho_\tau$ is also sufficient to handle the reaction terms while the bounds of corollary \ref{cor:apriorit} yields convergence of the discrete time derivatives.\\
For the initial data we employ the following trace inequality, cf. \cite{Schweizer2013}: For every $t\in[0,T]$  $p(\cdot,t),\,m(\cdot,t)\in L^2(\Omega)$ and $v\in L^2(0,T;H^1(\Omega))$ we have 
\begin{align*}
\left\Vert v(\cdot,t)\right\Vert_{L^2(\Omega)}\leq c\left\Vert v\right\Vert_{H^1(\Omega\times(0,t))}.
\end{align*}
Replacing $v$ by $p-p_0$ and $m-m_0$, we obtain
\begin{align*}
\left\Vert p(\cdot,t)-p_0\right\Vert_{L^2(\Omega)} &\leq c\left\Vert p-p_0\right\Vert_{H^1(\Omega\times(0,t))} \rightarrow 0 \quad \text{ as } \; t\rightarrow 0,\\
\left\Vert m(\cdot,t)-m_0\right\Vert_{L^2(\Omega)} &\leq c\left\Vert m-m_0\right\Vert_{H^1(\Omega\times(0,t))} \rightarrow 0 \quad \text{ as } \; t\rightarrow 0.
\end{align*}
This completes the proof of theorem \ref{thm:main} and ends this section.

\section{Asymptotic behaviour of solutions}

In the following we further discuss the asymptotics of the solutions, with respect to large-time as well as with respect to parameters in the system. We start with a discussion of stationary states, then proceed to stability issues and finally return to the connection with the Fisher-Kolmogorov equation.

\subsection{Stationary states}

System \eqref{eq:p}--\eqref{eq:m} features two constant stationary solutions, namely the trivial one $(\bar p_1,\bar m_1) = (0,0)$ and
\begin{align*}
 (\bar p_2, \bar m_2) = \left(\frac{(-\mu^2+(-q_m-q_p+\alpha)\mu+\alpha q_p)(\mu+q_p)}{\alpha(\mu+q_m+q_p)(\mu+q_m)}, -\frac{(\mu^2+(q_m+q_p-\alpha)\mu-\alpha q_p)q_m}{\alpha(\mu+q_m+q_p)(\mu+q_p)}\right).
\end{align*}
Interestingly, depending on the parameters, the state $(\bar p_2, \bar m_2)$ can be negative. Since the dynamics of the systems preserves the nonnegativity of solutions, this means that in these cases, this state cannot be reached. It is expected that the system then converges to $(0,0)$, meaning that the tumour will disappear. 

Examining the nominator of $\bar p_2$ (as $\bar m_2$ is a multiple of $\bar p_2$), we obtain a lower bound on the growth rate $\alpha$, i.e.
\begin{align}\label{eq:alphalower}
 \alpha \ge \frac{\mu^2 + \mu(q_p+q_m)}{\mu + q_p} = \mu\left(1 + \frac{q_m}{\mu + q_p}\right).
\end{align}
For $\mu = 0$ this is always satisfied while for $\alpha = \mu$ it is not, independent of $q_m$ and $q_p$ as long as $q_m > 0$. For large values of $\mu$, the expression becomes linear while for $q_p\to 0$ is will converge to infinity. In this situation, the motile cells cannot become proliferating ones and so the absorption will eventually destroy the tumour. See also figure \ref{fig:loweralpha} for a plot for some exemplary values of $q_p$ and $q_m$.
\begin{figure}
\begin{center}
  \includegraphics[width=.7\textwidth]{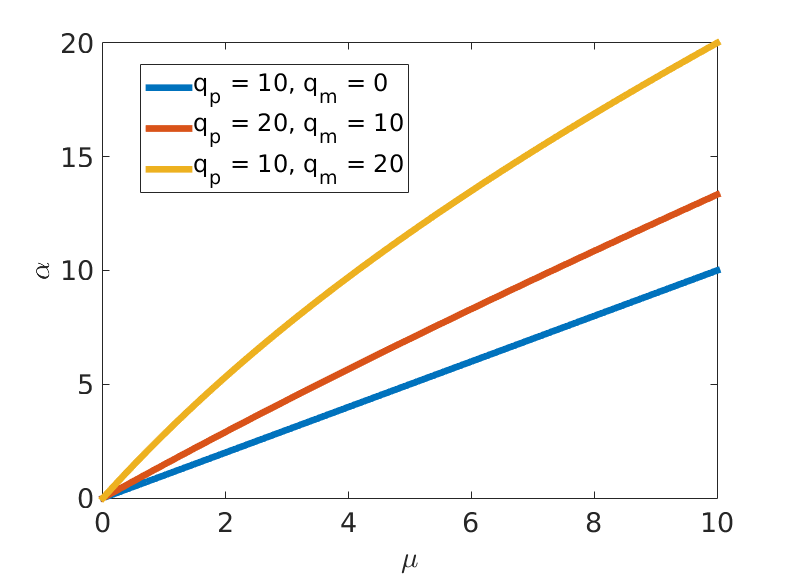}
\end{center}
 \caption{The lower bound of \eqref{eq:alphalower} as a function of $\mu$ for $(q_p,q_m) = (10,0)$, $(q_p,q_m) = (20,10)$ and $(q_p,q_m) = (10,20)$.}
 \label{fig:loweralpha}
\end{figure} 
%

\subsection{Linear stability}
 
With perturbations $\tilde p$, $\tilde m$, the linearized system around a constant stationary state $\bar p, \bar m$ is given by 
\begin{align*}
 \partial_t \tilde p &= D_\alpha (1-\bar p)\Delta  p - (\mu + q_m)\tilde p + q_p \tilde m + \alpha (1-\bar \rho)\tilde p - \alpha \bar p \tilde \rho,\\
 \partial_t \tilde m &= (1-\bar \rho)\Delta \tilde m + \bar m \Delta \tilde p - (\mu + q_p)\tilde m + q_m\tilde p.
\end{align*}
Let $-k^2$ be an eigenvalue of the Laplace operator, then for the perturbation in direction of the corresponding eigenvector we obtain a linear ordinary differential equation with system matrix
\begin{align*}A = 
\left(\begin{array}{cc}
 -D_\alpha k^2(1-\bar p)-  (\mu + q_m) + \alpha (1-\bar \rho) - \alpha \bar p  & q_p - \alpha \bar p \\
  -D_\nu k^2 \bar m+  q_m   & - D_\nu k^2  (1-\bar p) -(\mu + q_p)
       \end{array}\right).
\end{align*}
Hence, as usual for dynamical systems, linear stability is characterised by the negativity of the real parts of the eigenvalues of this system.
Consider first the trivial stationary state $(\bar p_1,\bar m_1) = (0,0)$, with the corresponding matrix
\begin{align*}A_1 = 
\left(\begin{array}{cc}
 -D_\alpha k^2 -  (\mu + q_m) + \alpha   & q_p   \\
   q_m   & - D_\nu k^2   -(\mu + q_p)
       \end{array}\right).
\end{align*}
Note that $\alpha =0$ and $\mu$ nonnegative it is straightforward to see that the eigenvalues have negative real values, hence the stationary state is stable. On the other hand for $\alpha$ large we see that $\frac{1}\alpha A_1$ can be seen as a perturbation of the matrix
\begin{align*}B_1 = 
\left(\begin{array}{cc}
 1  & 0  \\
   0   & 0
       \end{array}\right).
\end{align*}
which has a positive eigenvalue. Hence, as expected, for $\alpha$ sufficiently large the trivial stationary state becomes unstable.

For general parameters, its difficult to determine the sign of the eigenvalues, yet in the case when the trivial state is unstable, one expects a travelling wave solution to connect it to the nontrivial one. This has already been observed in \cite{Gerlee2012} and will be discussed in the next subsection.  

\subsection{Travelling wave solutions}
The model \eqref{eq:p}--\eqref{eq:m} can be understood as an extension to two different species of the classical Fisher-Kolmogorov-Equation, \cite{Murray2003}, with logistic non-linearity given as 
\begin{align*}
 \partial_t \rho = D\Delta \rho + \rho(1-\rho).
\end{align*}
It is well known that this equation features travelling wave solutions and therefore, already in \cite{Gerlee2012}, the existence of such solutions for \eqref{eq:p}--\eqref{eq:m} was examined. From a biological point of view these solutions are highly relevant since the wave speed corresponds to the growth rate of the tumour. We briefly repeat the discussion and start from the ansatz
\begin{align*}
p(x,t)=P(x-ct),\quad m(x,t)=M(x-ct) 
\end{align*}
which yields the following system of ODEs
\begin{align*}
P'=\,&Q \\
M'=\,&N \\
Q'=\,&\frac{2}{\alpha(1-P-M)}\big( \bar q_mP-q_pM-cQ-\alpha P(1-P-M)\big)\\
N'=\,&\frac{2}{\nu(1-P)}\Big( \bar q_pM-\frac{\nu M}{\alpha(1-P-M)}\big( \bar q_mP -q_pM-cQ\\
&-\alpha P(1-P-M) \big) -cN-q_mP \Big).
\end{align*}
supplemented with the boundary conditions
\begin{align*}
&P(-\infty)=\bar p_2,\quad M(-\infty)=\bar m_2,\quad Q(-\infty)=0, \quad N(-\infty)=0, \\
&P(\infty)=0, \quad M(\infty)=0,\quad Q(\infty)=0, \quad N(\infty)=0.
\end{align*}
corresponding to a healthy state with no tumour cells ($p_1$) and one with constant tumour density everywhere ($p_2$). To get an estimate on the wave speed, Garlee et al. \cite{Gerlee2012}, again perform a linear stability analysis with $c$ as a parameter and try to find the lowest value of $c$ such that the real part of all eigenvalues is negative. 
In the proceeding section, we will present several numerical examples complementing these analytic considerations.

\subsection{Scaling limit to Fisher-Kolmogorov}

Finally we comment further on the relation of the system to the Fisher-Kolmogorov equation for the total density. Indeed the latter can be obtained as a scaling limit. It is natural from the lattice-based derivation to assume that $D_\alpha \ll \alpha$, so we introduce a small parameter $\epsilon$ and assume $D_\alpha = \epsilon^\nu \tilde D_\alpha$ for $\nu > 0$ and $\tilde D_\alpha$ being at order one just as $\alpha$.  We further need a fast transition between the two states, so that they effectively become indistinguishable. Thus, we assume that $q_i = \frac{1}\epsilon  \tilde q_i$ for $i=m,p$, which leads to 
\begin{align}
\partial_t p&=\epsilon^\nu D_{\alpha}(1-\rho)\Delta p - \mu q  + \alpha p (1-\rho) +  \frac{1}\epsilon(\tilde q_p m - \tilde q_mp), \\
\partial_t m &= D_{\nu}\left((1-\rho)\Delta m +m \Delta \rho\right)  -\mu m - \frac{1}\epsilon(\tilde q_p m - \tilde q_mp).
\end{align}
We can consider a transformed system for the difference and the sum of these two equations. In the difference, there is a leading order $\frac{1}\epsilon$, which in the scaling limit yields $m= \lambda p$ with $\lambda = \frac{\tilde q_m}{\tilde q_p}$. Ignoring higher-order terms in $\epsilon$, the sum of the equations is given by 
$$ \partial_t \rho = D_{\nu}\left((1-\rho)\Delta m +m \Delta \rho\right)  -\mu \rho  + \alpha p (1-\rho). $$
Inserting $p=\frac{1}{1+\lambda} \rho$ and $p=\frac{\lambda}{1+\lambda} \rho$ we recover the Fisher-Kolmogorov equation
$$ \partial_t \rho = D \Delta \rho -\mu \rho  + \beta \rho (1-\rho), $$
with $D= D_\nu \frac{\lambda}{1+\lambda}$ and $\beta = \frac{\alpha}{1+\lambda}$.

\section{Numerics}
To discretise the system \eqref{eq:p}--\eq{eq:bcs}, we employ a conforming first order $H^1$ finite element discretisation, both in one and two spatial dimensions. To this end, we divide our domain into elements which are triangles in two space dimensions and intervals in 1D. For simplicity, we shall only discuss the two-dimensional case from now on. We cover the domain $\Omega \subset \RR^2$ by a finite collection of triangles which we denote by $\mathcal{T}_h$, where $h$ refers to the diameter of the largest triangle. On $\mathcal{T}_h$ we introduce the discrete space
\begin{align*}
 V_h = \{w \in H^1(\Omega) \;|\; w_{\left.\right|T}\in \mathcal{P}^1(T) \;\;\forall T \in \mathcal{T}_h \}
\end{align*}
where $\mathcal{P}^1(T)$ are the polynomials of degree one on $T$. We discretise in time, with stepsize $\tau>0$, using the following implicit-explicit (IMEX) scheme:
\begin{align*}
p_{k+1}&=p_k + \tau D_{\alpha}(1-\rho_k)\Delta p_{k+1} + R_p(p_k,m_k),\\
m_{k+1}&=m_k + \tau D_{\nu}\left((1-\rho_k)\Delta m_{k+1}+m_k\Delta\rho_{k+1}\right) + R_m(p_k,m_k).
\end{align*}
For our one dimensional examples we consider the domain $\Omega = [0,100]$, discretised into $600$ elements. In a first test we chose an initial datum with proliferating tumour cells $p$ both at the left boundary, i.e.
$$
p(x,0) = e^{-4x} + e^{-5(x-50)^2}\text{ and } m(x,0) = 0.
$$
The results are shown in figure \ref{fig:justwaves} and as expected, we obtain a travelling wave solution. In a second example, we put an additional bump of cells in the middle of the domain to observe the merging of two waves, see figure \ref{fig:meeting}. Finally, we perform a 2D simulation. Here, we chose a tumour with a hole to see how if and how it recovers, see figure \ref{fig:2dhole2}. We conclude by presenting two examples that are now covered by our present existence analysis, yet can be included at the cost of additional technicalities. In the first one, we introduce anisotropic diffusion to the motile cells, following the ideas presented in \cite{Painter2013}. More precisely, we rewrite equation \eqref{eq:m} in conservative form and introduce a space-dependent anisotropy which yields
\begin{align}\label{eq:mani}
 \partial_t m  = \nabla\cdot \left( D_\nu D_a(x) [ (1-\rho)\nabla m + m \nabla \rho)]\right) + R_m(p,m),
\end{align}
with 
\begin{align}\label{eq:ani_ex1}
 D_a(x) = \left(\begin{array}{cc}
                 \frac{1}{2} - d(x) & 0 \\
                 0 & \frac{1}{2} + d(x)
                \end{array}\right).
\end{align}
In our example presented in figure \ref{fig:2dani}, we constructed $d(x)$ by smoothing of uniformly distributed white noise, normalised such that $-0.5 < d(x) < 0.5$. For the second example, we consider 
\begin{align}\label{eq:ani_ex2}
 D_a(x) = (1-\delta_a) I_2 + \delta_a\gamma^T\gamma,
\end{align}
with $0 \le \delta_a \le 1$ and a given direction $\gamma$. By $I_2$, we denote the identity matrix in two dimensions. In figure \ref{fig:2dani2}, we present results for $\gamma = (1,1)^\mathrm{T}$ and different values of $\delta_a$ at time $t=100$.

\begin{figure}
\begin{center}
\begin{subfigure}[t]{0.325\textwidth}
      \centering
      \includegraphics[width=\textwidth]{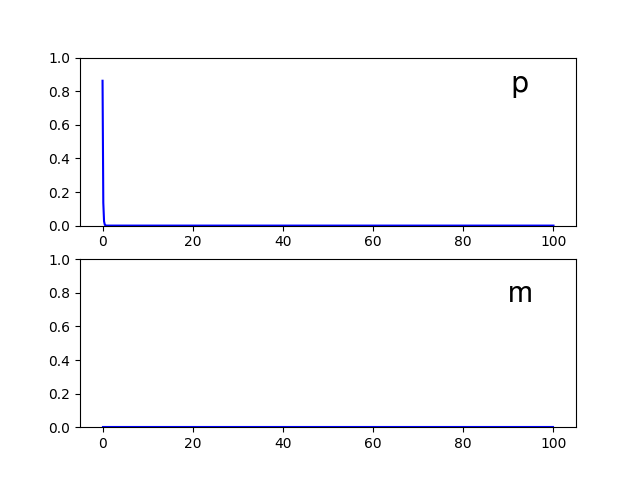}
      \caption{$t=0$}
  \end{subfigure}%
  \begin{subfigure}[t]{0.325\textwidth}
      \centering
      \includegraphics[width=\textwidth]{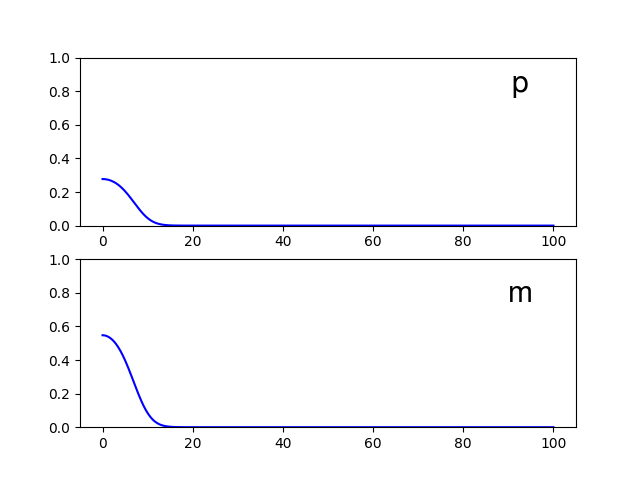}
      \caption{$t=1.7$}
  \end{subfigure}%
  \begin{subfigure}[t]{0.325\textwidth}
      \centering
      \includegraphics[width=\textwidth]{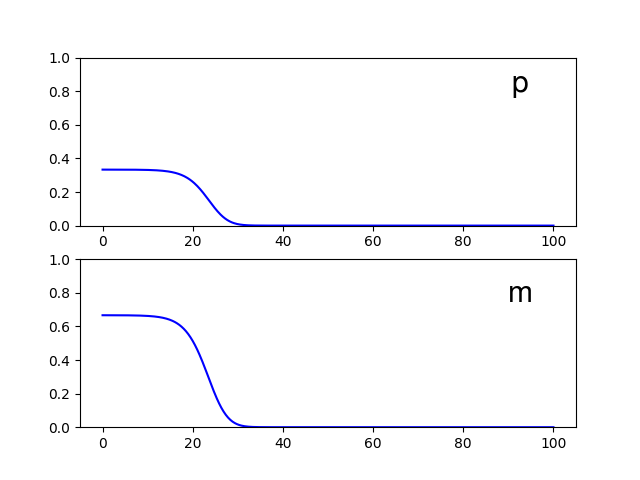}
      \caption{$t=3.9$}
  \end{subfigure}
 \end{center}
 \caption{Densities of $p$ (top row) and $m$ (bottom row) for parameters $\alpha =2$, $D_\alpha = D_\nu = 0.025$, $q_m=20$, $q_p=10$ and $\mu = 0$. The initial datum is given by $p(x,0) = e^{-10x}$ and $m(x,0)=0$}
 \label{fig:justwaves}
\end{figure} 

\begin{figure}
\begin{center}
\begin{subfigure}[t]{0.325\textwidth}
      \centering
      \includegraphics[width=\textwidth]{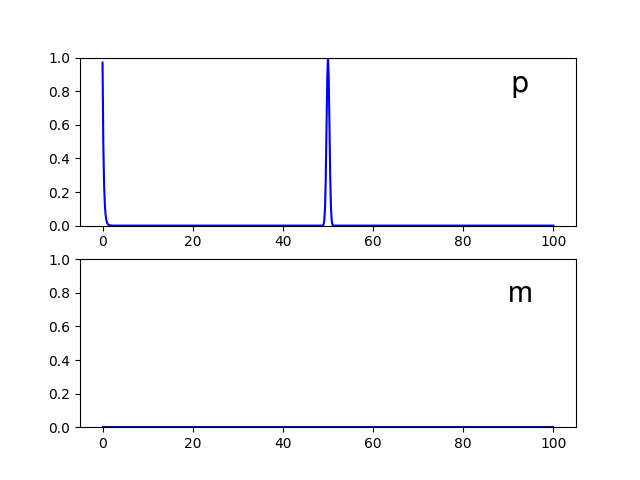}
      \caption{$t=0$}
  \end{subfigure}%
  \begin{subfigure}[t]{0.325\textwidth}
      \centering
      \includegraphics[width=\textwidth]{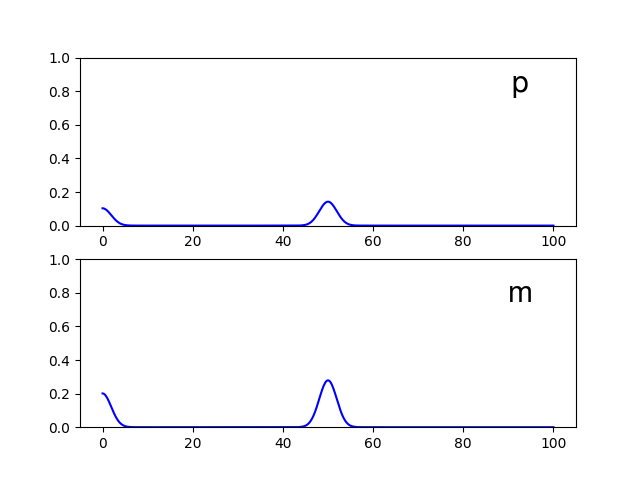}
      \caption{$t=3.5$}
  \end{subfigure}%
  \begin{subfigure}[t]{0.325\textwidth}
      \centering
      \includegraphics[width=\textwidth]{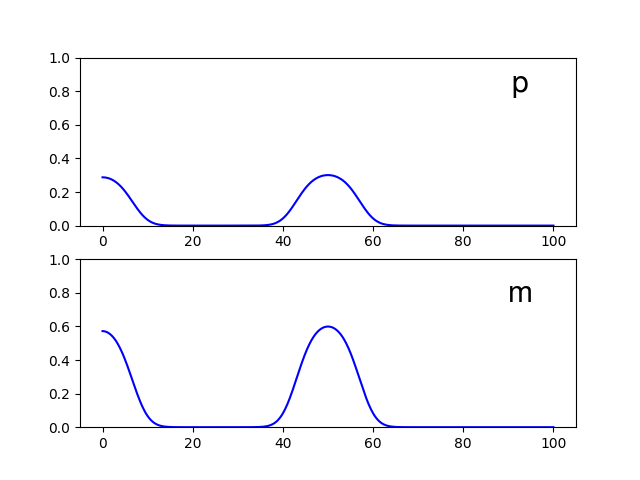}
      \caption{$t=14$}
  \end{subfigure}
  \\
  \begin{subfigure}[t]{0.325\textwidth}
      \centering
      \includegraphics[width=\textwidth]{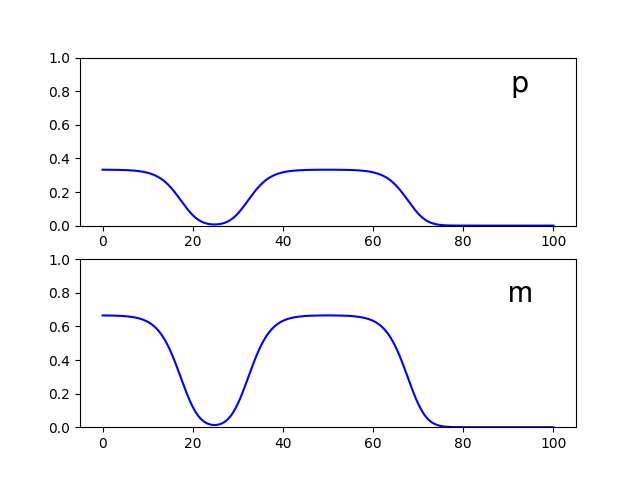}
      \caption{$t=29$}
  \end{subfigure}%
  \begin{subfigure}[t]{0.325\textwidth}
      \centering
      \includegraphics[width=\textwidth]{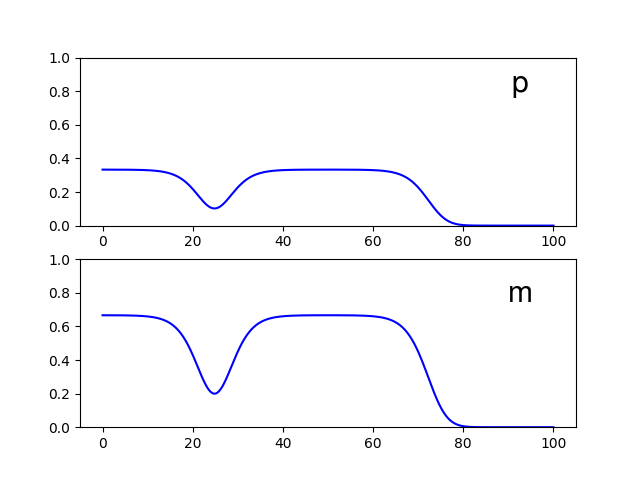}
      \caption{$t=35$}
  \end{subfigure}%
  \begin{subfigure}[t]{0.325\textwidth}
      \centering
      \includegraphics[width=\textwidth]{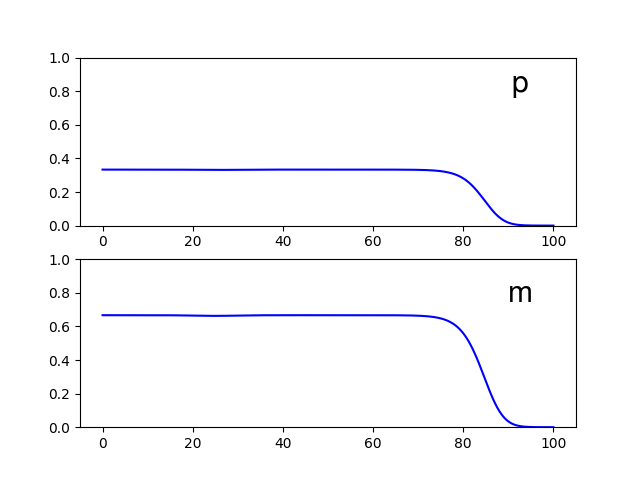}
      \caption{$t=51$}
  \end{subfigure}%
 \end{center}
 \caption{Illustration of two travelling waves meeting each other. Densities of $p$ (top row) and $m$ (bottom row) for parameters $\alpha = 1$, $D_\alpha = D_\nu = 0.5$, $q_m=20$, $q_p=10$ and $\mu = 0$. The initial datum is given by $p(x,0) = \exp(-4x) + 0.99\exp(-5(x-50)^2)$ and $m(x,0)=0$.}
 \label{fig:meeting}
\end{figure} 

\begin{figure}
\begin{center}
      \includegraphics[scale=.22]{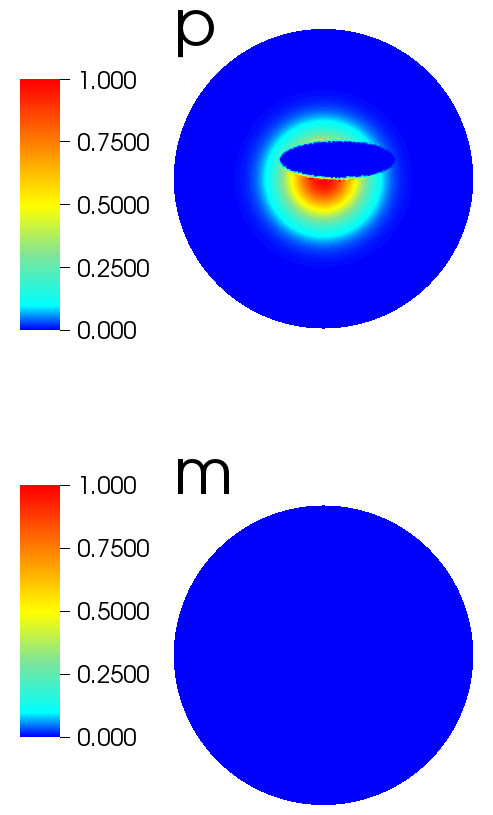}
      \includegraphics[scale=.22]{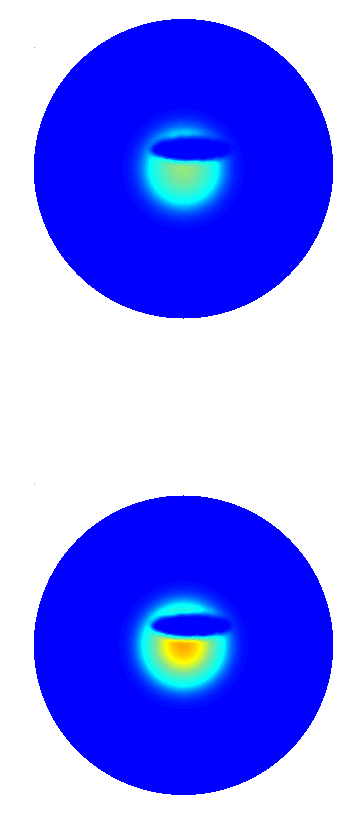}
      \includegraphics[scale=.22]{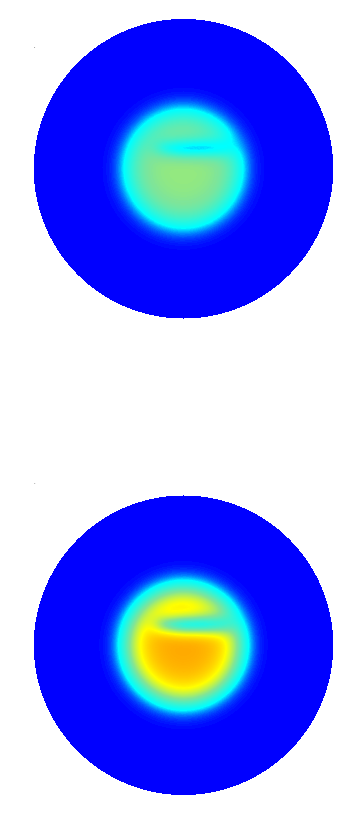}
      \includegraphics[scale=.22]{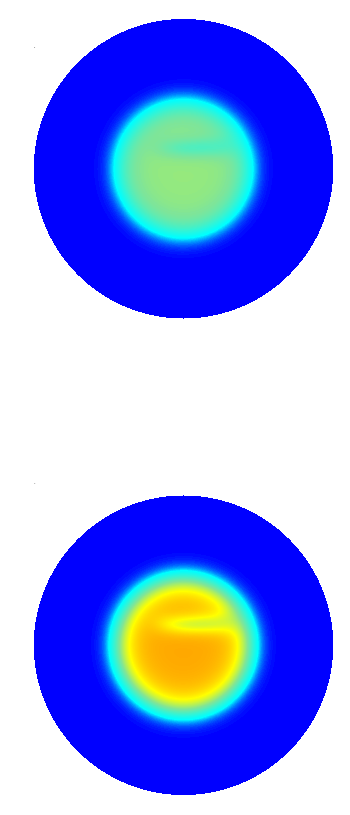}
      \includegraphics[scale=.22]{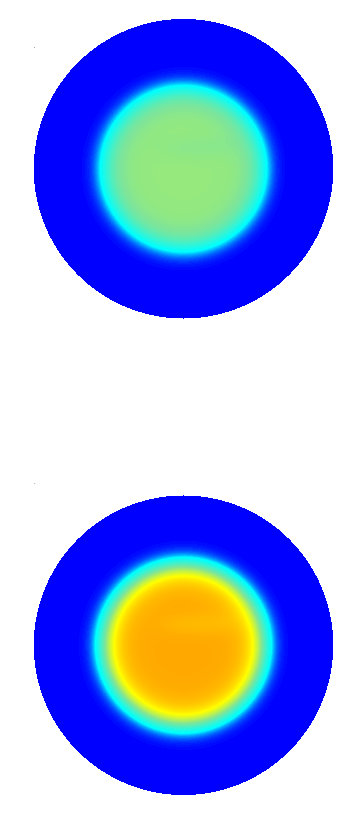}
      \begin{flushleft}
      \hskip6em $ t = 0$  \hskip4em $t = 0.05$ \hskip4em $t = 0.5$ \hskip4em $t = 0.7$ \hskip4.5em $t = 1$
      \end{flushleft}
  \end{center}
  
   \caption{Densities of $p$ (top row) and $m$ (bottom row) for parameters $\alpha =2$, $D_\alpha = D_\nu = 0.025$, $q_m=20$, $q_p=10$ and $\mu = 0$. The initial datum is given by $p(x,0) = e^{-4x} + e^{-5(x-50)^2}$ and $m(x,0)=0$.}
 \label{fig:2dhole2}
\end{figure} 

\begin{figure}
\begin{center}
      \includegraphics[scale=.22]{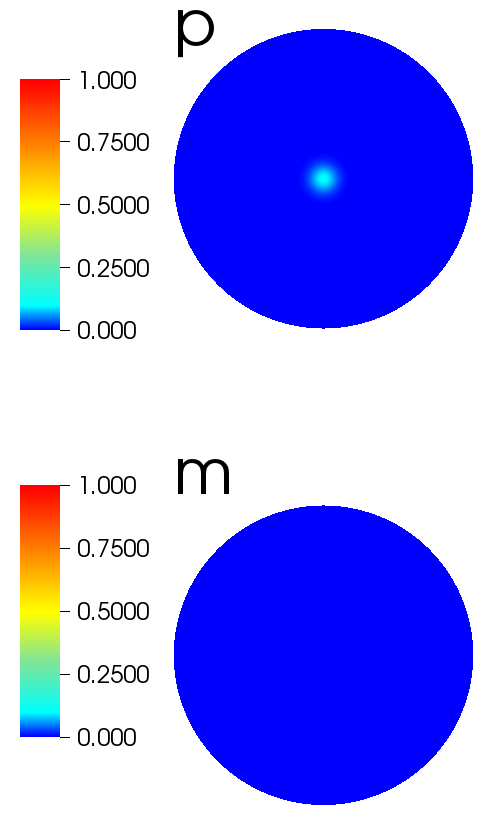}
      \includegraphics[scale=.22]{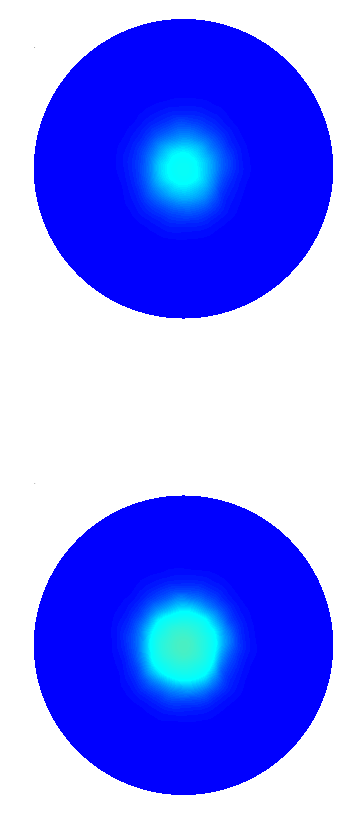}
      \includegraphics[scale=.22]{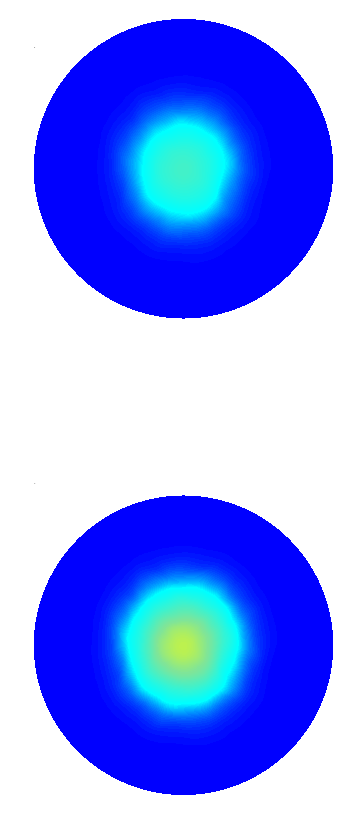}
      \includegraphics[scale=.22]{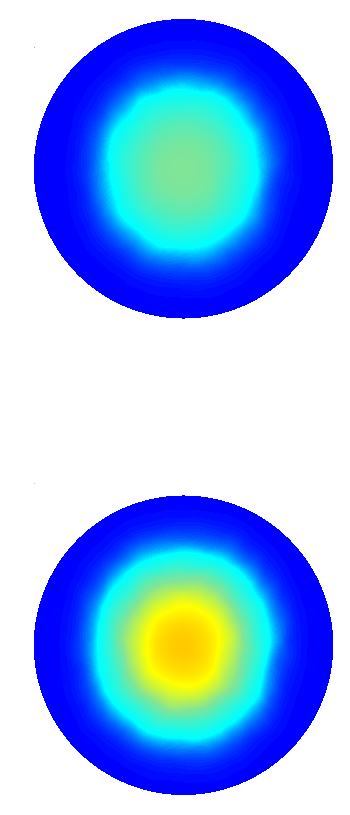}
      \includegraphics[scale=.22]{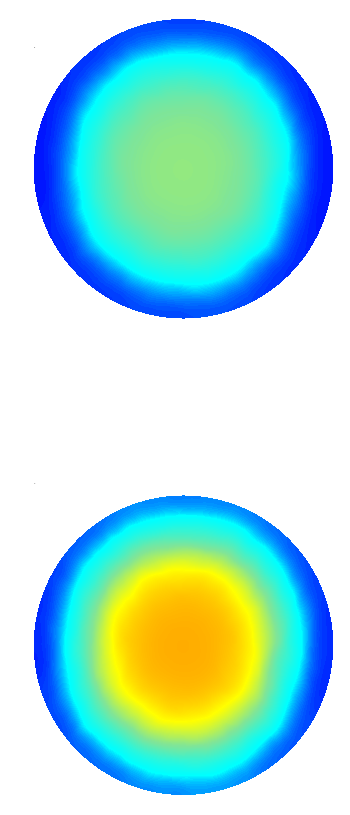}
      \begin{flushleft}
      \hskip6em $ t = 0$  \hskip4em $t = 0.5$ \hskip4em $t = 0.7$ \hskip5em $t = 1$ \hskip4.5em $t = 1.3$
      \end{flushleft}
  \end{center}
  
   \caption{Densities of $p$ (top row) and $m$ (bottom row) for parameters $\alpha =2$, $D_\alpha = 0.0005,\, D_\nu = 0.05$, $q_m=20$, $q_p=10$ and $\mu = 0$ and with the anisotropic equation \eqref{eq:mani} and $D_a$ as in \eqref{eq:ani_ex1}. The initial datum is given by $p(x,0) = 0.1e^{-x^2+y^2}$ and $m(x,0)=0$.}
 \label{fig:2dani}
\end{figure}

\begin{figure}
\begin{center}
      \includegraphics[scale=.22]{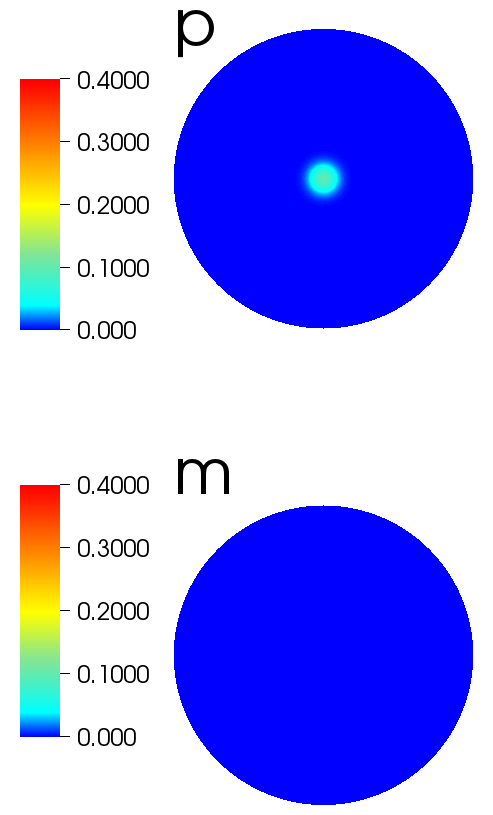}
      \includegraphics[scale=.22]{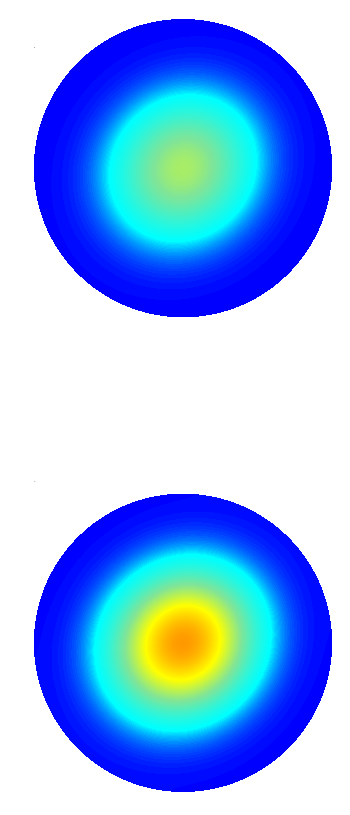}
      \includegraphics[scale=.22]{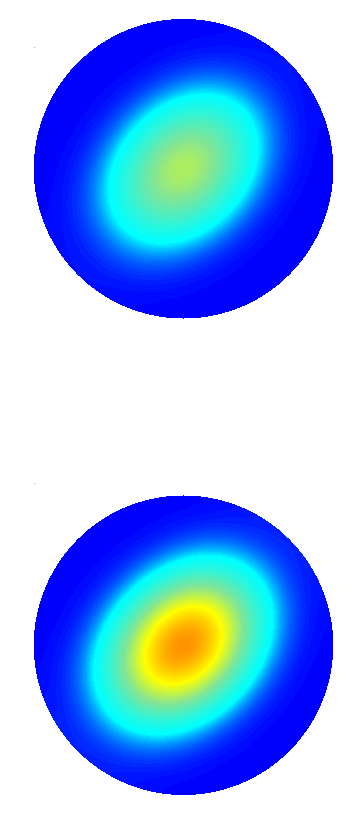}
      \includegraphics[scale=.22]{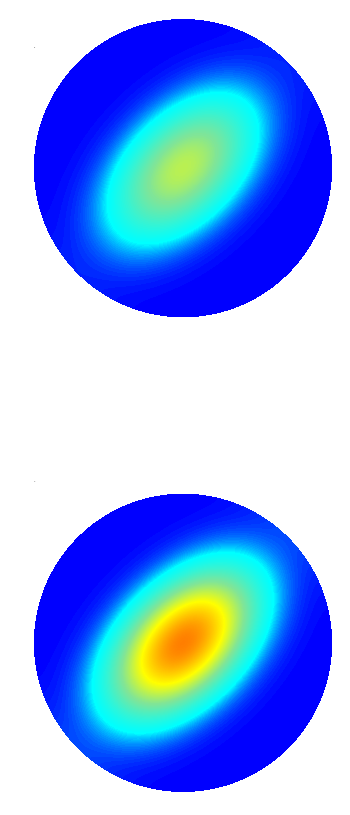}
      \includegraphics[scale=.22]{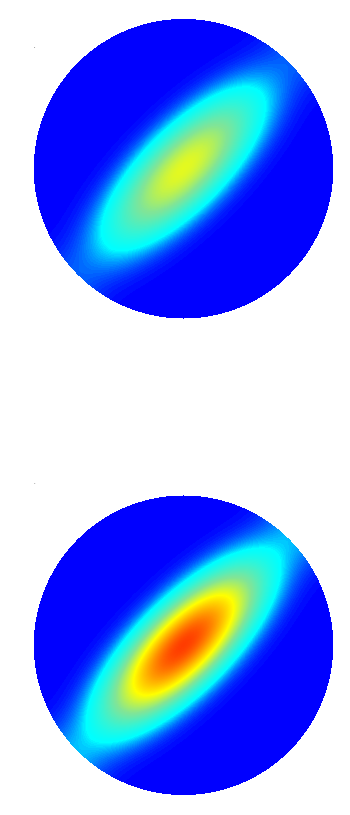}
      \begin{flushleft}
      \hskip3em initial datum  \hskip3.5em $\delta_a = 0.1$ \hskip3.5em $\delta_a = 0.3$ \hskip3.5em $\delta_a = 0.5$ \hskip4em $\delta_a = 0.8$
      \end{flushleft}
  \end{center}
  
   \caption{Densities of $p$ (top row) and $m$ (bottom row) for parameters $\alpha =2$, $D_\alpha = 0.0005,\, D_\nu = 0.05$, $q_m=20$, $q_p=10$ and $\mu = 0$ and with the anisotropic equation \eqref{eq:mani} and $D_a$ as in \eqref{eq:ani_ex2}. The initial datum is given by $p(x,0) = 0.1e^{-x^2+y^2}$ and $m(x,0)=0$.}
 \label{fig:2dani2}
\end{figure}

\section{Outlook}

There are several interesting questions for further work to be tackled based on our results. From an applied point of view it would be most interesting to incorporate realistic values of reaction rates and anisotropic diffusion locally resolved in neural tissue in order to obtain more quantitative predictions of the growth (cf. \cite{konukoglu2007towards}). Moreover, it would be a natural step to include the modelling of medical therapy such as radio- or chemotherapy to better understand their impact (cf. \cite{konukoglu2010extrapolating}). 

From a mathematical point an analysis of the multi-dimensional version of the reaction-cross-diffusion system would be a key step. We mention that our existence proof is basically independent of the dimension except for passing to the limit as $\tau \rightarrow 0$. In multiple dimensions we cannot expect to obtain a bound on the gradient of $(1-\rho_\tau) \nabla p_\tau$, but only its divergence. Thus, in passing to the limits in the scalar products of gradients we will not be able to employ any kind of compact embedding guaranteeing strong convergence of one of them. However, these terms might be just suitable to pass to the limit using compensated compactness, which we leave to further research. Another interesting question for the analysis is to further quantify the effect of the nonlinear cross-diffusion on the long-term behaviour when compared to the reaction system or Fisher-Kolmogorov.

\section*{Acknowledgements}
MB acknowledges support by the ERC via Grant EU FP 7 - ERC Consolidator Grant 615216 LifeInverse. The work of JFP was supported by DFG via Grant 1073/1-2. The authors thank Prof. Thomas Hillen (Alberta) for useful discussions.

\bibliographystyle{plain}
\bibliography{glio}

\end{document}